\documentclass[11pt]{article}       
\usepackage{geometry}               
\usepackage{graphicx}
\usepackage{caption}
\usepackage{subcaption}
\usepackage{enumerate}
\usepackage{hyperref}
\usepackage{cite}

\usepackage{graphics,color} 
\usepackage{epsfig} 
\usepackage{amsmath,mathtools,booktabs,multirow,tabularx,tablefootnote}
\usepackage{amssymb}
\newcommand{\onenorm}[1]{\left\lvert#1\right\rvert}

\usepackage{enumitem}
\usepackage[normalem]{ulem}

\usepackage{amsfonts}
\usepackage{epstopdf}
\usepackage{caption}
\usepackage{amsmath,amsfonts,mathrsfs,amssymb,amscd}
\usepackage{amsthm}

\usepackage{mathtools}
\usepackage{bm}
\usepackage{lipsum}
\usepackage{algorithm}
\usepackage{algpseudocode}

\usepackage{color}
\usepackage{array}
\usepackage{multirow}
\usepackage{cite}
\usepackage{url}
\usepackage{booktabs}

\newtheorem{proposition}{Proposition}

\newtheorem{theorem}{Theorem}
\newtheorem{remark}{Remark}

\usepackage[scr]{rsfso}

\DeclareMathOperator*{\argminA}{arg\,min} 
\DeclareMathOperator*{\argmaxA}{arg\,max}

\title{Using Affine Policies to Reformulate Two-Stage Wasserstein Distributionally Robust Linear Programs to be Independent of Sample Size\thanks{This work was supported in part by  the National Research Foundation of Korea funded by  MSIT(2020R1C1C1009766, 2021R1A4A2001824), the Information and Communications Technology Planning and Evaluation  grant funded by MSIT(2022-0-00480), and Samsung Electronics. 
A preliminary version of this work was presented at the 61st IEEE Conference on Decision and Control~\cite{cho2022affine}.} }

\author{Youngchae Cho \and
 Insoon Yang\thanks{A. Hakobyan, and I. Yang are with the Department of Electrical and Computer Engineering and ASRI, Seoul National University, Seoul, 08826, South Korea {\tt\small \{youngchaecho, insoonyang\}@snu.ac.kr}}
}

\date{}

\begin{document}
\maketitle

\pagestyle{myheadings}
\thispagestyle{plain}

\begin{abstract}
Intensively studied in theory as a promising data-driven tool for decision-making under ambiguity, two-stage distributionally robust optimization (DRO) problems over Wasserstein balls are not necessarily easy to solve in practice. This is partly due to large sample size. In this article, we study a generic two-stage distributionally robust linear program (2-DRLP) over a 1-Wasserstein ball using an affine policy. The 2-DRLP has right-hand-side uncertainty with a rectangular support. Our main contribution is to show that the 2-DRLP problem has a tractable reformulation with a scale independent of sample size. The reformulated problem can be solved within a pre-defined optimality tolerance using robust optimization techniques. To reduce the inevitable conservativeness of the affine policy while preserving  independence of sample size, we further develop a method for constructing an uncertainty set with a probabilistic guarantee over which the Wasserstein ball is re-defined. As an application, we present a novel unit commitment model for power systems under uncertainty of renewable energy generation to examine the effectiveness of the proposed 2-DRLP technique. Extensive numerical experiments demonstrate that our model leads to better out-of-sample performance on average than other state-of-the-art distributionally robust unit commitment models while staying computationally competent. 
\end{abstract}

\section{Introduction}\label{sec:intro}
Two-stage optimization is a popular tool for sequential decision-making under uncertainty, where the decision maker makes two kinds of decisions, i.e., here-and-now and wait-and-see decisions, before and after observing the realization of uncertainty, respectively.  Due to its generality, two-stage optimization has seen many applications in various research fields such as inventory management \cite{nikzad2019two}, workforce management \cite{mattia2017staffing}, location planning \cite{oksuz2020two}, and power system operations \cite{bertsimas2013adaptive,zhao2019two}. In the present article, we consider a class of two-stage optimization problems based on distributionally robust optimization (DRO) with the Wasserstein metric.

\subsection{Backgrounds}\label{sec:backgrounds}
Two-stage optimization approaches can be conveniently classified by the stochastic optimization method. Among the most-studied stochastic optimization methods for two-stage optimization are stochastic programming (SP), robust optimization (RO) and DRO. A usual objective of SP is to minimize the expected total cost, i.e., a sum of the deterministic cost associated with here-and-now decisions and the expected cost associated with wait-and-see decisions, with respect to a probability distribution of uncertainty \cite{shapiro1998simulation}. As the true distribution of uncertainty is difficult to obtain, an empirical distribution is used instead in most cases. For this reason, SP works well only with large sample datasets. Without struggling to acquire the true distribution, RO uses worst-case analyses over an {\it uncertainty set} (a set of possible scenarios of uncertainty) with the common aim of minimizing the worst-case total cost, i.e., a sum of the deterministic cost associated with here-and-now decisions and the worst-case cost associated with wait-and-see decisions \cite{ben2004adjustable}. However, RO is often overly conservative as it ignores probabilistic features of uncertainty, which can be partially obtained through samples.

To mitigate the disadvantages of SP and RO simultaneously, DRO uses worst-case analyses for an {\it ambiguity set}, i.e., a family of probability distributions of uncertainty. A typical goal of DRO is to minimize the expected total cost with respect to worst-case distributions in an ambiguity set. Incorporating probabilistic features while hedging against the potential inappropriateness of any single pre-specified distribution, DRO better balances efficiency and robustness compared to SP and RO. For details of general DRO problems, see, for example, \cite{rahimian2019distributionally} and the references therein. 

Performances of DRO greatly depend on how the ambiguity set is chosen. 
For example, ambiguity sets can be defined using $f$-divergences \cite{bayraksan2015data}, e.g., the Kullback--Leibler (KL) divergence \cite{jiang2016data} and the total variation distance \cite{sun2016convergence}, as well as moment conditions \cite{delage2010distributionally,bertsimas2019adaptive}. However, these ambiguity sets have a few limitations. First, ambiguity sets based on $f$-divergences may not be rich enough as they include only distributions that are absolutely continuous with respect to a nominal distribution. Moreover, the underlying assumption of moment information known a priori for DRO based on moment conditions hardly seems justifiable \cite{delage2010distributionally}. Reportedly, moment-based DRO solutions may also be overly conservative~\cite{wang2016likelihood}.

Ambiguity sets can be constructed using the Wasserstein metric as well \cite{kuhn2019wasserstein,gao2022distributionally}. A {\it Wasserstein ball} is defined as a statistical ball in the space of probability distributions, the radius of which is measured using the Wasserstein metric. Intuitively, the Wasserstein distance of two distributions is interpreted as the minimum cost of redistributing the probability mass from one distribution to the other. The center of a Wasserstein ball is mostly an empirical distribution constructed with a finite number of samples. As the elements of a Wasserstein ball are perturbations of the nominal distribution that are obtained considering the distance of uncertain scenarios, Wasserstein DRO does not suffer from the aforementioned drawbacks of DRO based on $f$-divergences or moment conditions. Moreover, Wasserstein DRO offers a strong finite-sample performance guarantee~\cite{esfahani2018data}. For these reasons, we focus on two-stage Wasserstein DRO in this article. 

\subsection{Related Work}
Research works providing solution methods for two-stage Wasserstein DRO in general forms includes \cite{esfahani2018data,zhao2018data,gamboa2021decomposition,wang2021second,duque2022distributionally,hanasusanto2018conic,bertsimas2022two,xie2020tractable,gangammanavar2022stochastic,byeon2022two,bansal2018decomposition,kim2020dual} 
all of which, except for \cite{esfahani2018data}, consider linear costs of here-and-now and wait-and-see decisions. Specifically, \cite{zhao2018data,gamboa2021decomposition,wang2021second,duque2022distributionally,hanasusanto2018conic,bertsimas2022two,xie2020tractable,gangammanavar2022stochastic} deal with two-stage distributionally robust linear programs (2-DRLPs) over Wasserstein balls, where the second-stage problem to optimize wait-and-see decisions is a linear program (LP) while here-and-now decision variables can be integer or continuous. In \cite{zhao2018data}, it is briefly mentioned that 2-DRLPs over 1-Wasserstein balls can be reformulated as tractable semi-infinite or finite-dimensional optimization problems if the 1-, 2- or $\infty$-norm is used as the metric on the support. In \cite{gamboa2021decomposition}, decomposition algorithms are developed for solving exact reformulations of 2-DRLPs over 1-Wasserstein balls with the 1- and $\infty$-norm, assuming right-hand-side uncertainty and a rectangular uncertainty set. The algorithms build on Benders decomposition \cite{rahmaniani2017benders} and the column-and-constraint generation (C\&CG) method \cite{zeng2013solving}. In \cite{wang2021second}, a second-order conic programming approach is employed to derive tractable reformulations of 2-DRLPs over 1-Wasserstein balls with the $2$-norm, assuming that uncertainty appears in either the objective function or the constraints. In \cite{duque2022distributionally}, cutting-plane algorithms are used to exactly solve 2-DRLPs over 1-Wasserstein balls with either the generic $p$-norm for $p\geq1$ or a class of quadratic functions. In \cite{hanasusanto2018conic}, 2-DRLPs with the Wasserstein metric of order 2 are exactly solved using conic programming approaches. In \cite{bertsimas2022two}, 2-DRLPs over $\infty$-Wasserstein balls with the $p$-norm are approximately solved by applying multiple decision policies, one for an uncertainty set associated with each sample data point. This scheme achieves optimality asymptotically, i.e., as the number of samples goes to infinity. In \cite{xie2020tractable}, tractable reformulations of 2-DRLPs over $\infty$-Wasserstein balls with uncertainty in either the objective function or the constraints are presented for different continuity conditions on the uncertainty. In \cite{gangammanavar2022stochastic}, a sequential algorithm is developed for general two-stage DRO problems and applied to 2-DRLPs over 1- and $\infty$-Wasserstein balls for demonstration. This algorithm creates at each iteration a Wasserstein ball using only a finite subset of the support as an approximation to the original ambiguity set. With a new observation added at each iteration, the algorithm is proved to achieve asymptotic optimality. 

References \cite{byeon2022two,bansal2018decomposition,kim2020dual,esfahani2018data} address more general classes of two-stage Wasserstein DRO problems than 2-DRLPs. In \cite{byeon2022two}, two-stage distributionally robust conic LPs over 1-Wasserstein balls are considered, for which a cutting-plane algorithm based on Benders decomposition is suggested. In \cite{bansal2018decomposition} and \cite{kim2020dual}, decomposition methods are developed assuming that both here-and-now and wait-and-see decisions are at least partially binary. The authors of \cite{esfahani2018data} study a class of two-stage DRO problems over 1-Wasserstein balls where the costs of wait-and-see decisions are written as point-wise maximums of finitely many concave functions of uncertainty. Using main results, tractable reformulations of 2-DRLPs over 1-Wasserstein balls with uncertainty in either the objective function or the right-hand-side of constraints are presented in \cite{esfahani2018data}. 

Notably, most of the existing solution methods for two-stage Wasserstein DRO problems, including those suggested in \cite{zhao2018data,esfahani2018data,gamboa2021decomposition,wang2021second,duque2022distributionally,hanasusanto2018conic,bertsimas2022two,xie2020tractable,gangammanavar2022stochastic,byeon2022two,bansal2018decomposition,kim2020dual}, have a scalability issue regarding sample size, i.e., the number of historical sample data. In other words, the existing solution methods require more computational resources for more samples. This implies that two-stage Wasserstein DRO problems may not yield desired solutions that fully exploit historical data at hand when computational resources are limited. 

\subsection{Contributions}
In this article, we study a generic 2-DRLP over a 1-Wasserstein ball, which has right-hand-side uncertainty with a rectangular support, using an affine policy. Affine policies are a frequently used solution method for two-stage optimization problems which impose the linear dependence of wait-and-see decisions on uncertain parameters. First developed in the context of SP \cite{holt1955linear,charnes1958cost,garstka1974decision}, affine policies had been disregarded by the operations research community due to their intrinsic conservativeness that is hard to meaningfully quantify \cite{georghiou2021optimality}. 
A few decades later, however, affine policies have gained wide attention in the fields of not only SP \cite{bodur2018two} but also RO \cite{ben2004adjustable,bertsimas2015performance,el2021optimality} as well as control theory for dynamical systems  \cite{bemporad2003min,kerrigan2003robust,skaf2010design,bertsimas2010optimality,hadjiyiannis2011efficient,zhang2017robust} due to their superior tractability and desirable properties related to cost performances such as robust invariance \cite{bertsimas2012power}. 
Not only studied in theory, affine policies have seen many applications thereafter as well, e.g., in portfolio management \cite{calafiore2008multi,fonseca2012international} and power system operations \cite{lorca2016multistage,duan2018distributionally,dehghan2019robust,ratha2020affine}.
Furthermore, researchers have successfully extended these approaches by using piecewise affine \cite{ben2020tractable,thoma2022designing},
segregated affine \cite{chen2008linear,chen2009uncertain} 
and polynomial \cite{bampou2011scenario}
policies. 

The main contributions of this study are three-fold. First, we show that the 2-DRLP of our interest has a tractable reformulation with a scale independent of sample size. For this, we first recast the worst-case expectation problem nested in it, which is infinite-dimensional, as a finite convex program with a scale that grows with sample size. We then aggregate optimization variables associated with different sample indices, which intuitively represent perturbation of samples, exploiting the fact that they have the same cost coefficient due to the affine policy. This yields an LP equivalent to the nested infinite-dimensional program, the scale of which is invariant with sample size. Finally, using duality in LPs, we obtain a finite-dimensional mixed-integer LP (MILP) as an exact reformulation of the 2-DRLP.
The reformulated problem can be solved up to a pre-defined precision by RO techniques. We also present a cutting-plane algorithm for the reformulated problem. As a result, many samples can be efficiently exploited without relying on computationally expensive decomposition algorithms. To the best of our knowledge, our study is the first to reveal that affine policies can resolve the scalability issue regarding sample size in a general class of two-stage Wasserstein DRO problems.\footnote{Although the scalability issue is addressed by \cite{zhu2019wasserstein} for the unit commitment (UC) problem, the method in \cite{zhu2019wasserstein} is  applicable only when the cost of wait-and-see decisions calculated using an affine policy is univariate. In contrast, we do not impose any special assumption on the affine policy.}

Meanwhile, the optimality gap incurred by the affine policy can be arbitrarily large when the size of the Wasserstein ambiguity set is big enough. We assert that it is also true for any value of the radius, because the optimality gap as a function of the radius is a difference of two concave functions, which can be neither increasing nor decreasing in general. To reduce the inevitable conservativeness of the affine policy, we re-define the Wasserstein ball on an uncertainty set smaller than the support. Our second main contribution is to design a data-driven method for constructing an uncertainty set with a bounded worst-case confidence level, over which the Wasserstein ball is rebuilt. Since the feasibility of the affine policy is guaranteed on a smaller uncertainty set, more efficient solutions can be obtained by using our method. Unlike existing data-driven methods for building an uncertainty set with a similar probabilistic guarantee, our method ensures that  the 2-DRLP  does not depend on sample size. 

Finally, to illustrate the applicability and effectiveness of the 2-DRLP approach using an affine policy for practical decision-making problems, we develop a novel UC model for power systems under the uncertainty of renewable generation. Extensive numerical experiments demonstrate that the proposed UC model outperforms not only classical models based on SP and RO but four state-of-the-art models based on DRO using ambiguity sets with the moment conditions \cite{zhou2019distributionally}, KL divergence \cite{chen2018distributionally}, 1-norm distance \cite{ding2018duality} and cumulative density function (CDF) \cite{duan2017data} in terms of out-of-sample performance, while staying computationally competent. 

The rest of this article is organized as follows. In Section \ref{sec:problem}, we formulate the 2-DRLP of our interest. In Section \ref{sec:appx}, we show that the 2-DRLP has a tractable reformulation with a scale independent of sample size. Furthermore, we provide a cutting-plane algorithm for solving the reformulated problem. In Section \ref{sec:support}, we explain how to construct an uncertainty set with a probabilistic guarantee, over which we rebuild the Wasserstein ball to reduce conservativeness. In Section \ref{sec:UC}, we present the novel UC model based on the 2-DRLP approach using an affine policy and discuss simulation results. In Section \ref{sec:conclusions}, we give concluding remarks. 

\noindent\textbf{Notation.} We denote by $\mathbb R$, $\mathbb R_+$, and $\mathbb R_-$ the sets of all real numbers, non-negative real numbers, and non-positive real numbers, respectively. For a natural number $n$, we denote by $1_n$, $0_n$, $I_n$, and $O_n$ the vector of ones, vector of zeros, identical matrix, and square zero matrix, respectively, all of dimension $n$. Furthermore, $\left[\cdot\right]_n$ represents the $n$th entry of a vector. We use $\onenorm{\,\cdot\,}$ to denote the 1-norm of a vector or the cardinality of a finite set. We also denote by $\left(\cdot\right)^\top$, $\delta_{\left(\cdot\right)}$, $\mathbb E$, $\circ$, $\left(\cdot\right)^{\circ}$, and ${\mathcal V}\left(\cdot\right)$ the transpose of a vector or matrix, Dirac delta distribution centered at a given point, expectation operator, entrywise product operator for two vectors, interior of a subset of a Euclidean space, and vertex set of a convex polytope, respectively.

\section{Problem Formulation}\label{sec:problem}
In this section, we formulate a two-stage Wasserstein DRO problem of our interest using an affine policy. To this end, we first consider the 2-DRLP 
\begin{equation}\label{eq:problem}
\min_{x_1\in{\mathcal X}_1}c^\top_1 x_1 + \max_{{\mathbb P}\in{\mathcal P}_{\varepsilon}\left(\Xi\right)}{\mathbb E}_{\mathbb P}\left[f\left(x_1,\xi\right)\right]
\end{equation}
where
\begin{equation}\label{eq:f}
f\left(x_1,\xi\right):=\min_{x_2\in{\mathcal X}_2\left(x_1,\xi\right)} c_2^\top x_2
\end{equation}
denotes the optimal cost of wait-and-see decisions. Here, $\xi\in{\mathbb R}^m$ and $\Xi\subset{\mathbb R}^m$ denote a random vector and its support, respectively. The support $\Xi$ is a bounded box that is known, i.e., $\Xi=[\underline{\xi},\overline{\xi}]$ where $\underline{\xi},\overline{\xi}\in{\mathbb R}^m$ can be obtained using a priori knowledge. We assume that $N$ historical samples ${\xi}_1,\ldots,{\xi}_N$ of ${\xi}$ are available and denote the index set of samples by ${\mathcal I}:=\left\{1,\ldots,N\right\}$.  

In (\ref{eq:problem}), $x_1\in\left\{0,1\right\}^{n_{11}}\times{\mathbb R}^{n_{12}}$ and $c_1\in{\mathbb R}^{n_{1}}$ with $n_1:=n_{11} + n_{12}$ represent a here-and-now decision vector and its cost coefficient vector, respectively. The feasible set ${\mathcal X}_1$ of $x_1$ is defined with finitely many linear inequalities. The symbol ${\mathcal P}_{\varepsilon}\left(\cdot\right)$ denotes a 1-Wasserstein ball on a given uncertainty set, which is a ball of radius $\varepsilon>0$ centered at an empirical distribution ${\mathbb P}_{\rm e}:=\frac{1}{N}\sum_{i\in{\mathcal I}}\delta_{{\xi}_i}$ in the space of probability distributions supported on the given uncertainty set. Specifically, we let 
\[
{\mathcal P}_{\varepsilon}\left(\cdot\right):= \left\{{\mathbb P}\in{\mathcal P}\left(\cdot\right): d\left({\mathbb P},{\mathbb P}_{\rm e}\right)\leq\varepsilon\right\}
\]
where ${\mathcal P}\left(\cdot\right)$ represents the family of all probability distributions supported on a given uncertainty set. 
Furthermore, $d$ denotes the Wasserstein metric of order 1 defined with the 1-norm, i.e.,  
\[
d\left({\mathbb P},{\mathbb P}^\prime\right):=\inf_{\pi\in\Pi\left({\mathbb P},{\mathbb P}^\prime\right)}\int_{{\Xi}\times{\Xi}}\onenorm{{\xi} - {\xi}^\prime}\pi\left(d{\xi},d{\xi}^\prime\right)
\] 
where $\Pi\left(\cdot,\cdot\right)$ denotes the set of all joint distributions supported on $\Xi\times\Xi$ with marginals equal to two given distributions. In (\ref{eq:f}), $x_2\in{\mathbb R}^{n_2}$ and $c_2\in{\mathbb R}^{n_2}$ represent a wait-and-see decision vector and its cost coefficient vector, respectively. The feasible set of $x_2$ is defined as
\[
{\mathcal X}_2\left(x_1,\xi\right):=\left\{x_2\in{\mathbb R}^{n_2}: A^{\rm in}_1 x_1 + A^{\rm in}_2 x_2 + A^{\rm in}_3{\xi} \leq  b^{\rm in}\right\}
\]
where $A^{\rm in}_1\in{\mathbb R}^{L\times n_1}$, $A^{\rm in}_2\in{\mathbb R}^{L\times n_2}$, $A^{\rm in}_3\in{\mathbb R}^{L\times m}$, and $b^{\rm in}\in{\mathbb R}^{L}$. In the above formulations, $m$, $n_{11}$, $n_2$, and $L$ are natural numbers, while $n_{12}$ is a non-negative integer. Throughout the study, we assume that (\ref{eq:problem}) is feasible, as is standard in the literature \cite{bertsimas2012power}. However, we do not impose the (relative) complete recourse condition, which is also usual  (see, e.g., \cite{bansal2018decomposition,byeon2022two}) but might be restrictive for some real-world problems \cite{bertsimas2022two}.

Problem (\ref{eq:problem}) is general enough to model diverse decision-making problems in the real world. For example, the biomass network design \cite{ning2019data}, unmanned aerial vehicle network design \cite{hou2021integrated}, and railway scheduling \cite{liu2022data} problems have been addressed in the form of (\ref{eq:problem}). However, it is often computationally demanding to exactly solve a two-stage optimization problem such as (\ref{eq:problem}) \cite{feige2007robust}. 

In this article, we focus on affine policies that approximately solve (\ref{eq:problem}). Affine policies are a popular solution method for two-stage optimization problems, where wait-and-see decision variables are restricted to be affine functions of uncertainty. Due to their computational efficiency, affine policies have been studied extensively for practical two-stage RO \cite{ben2011robust,lorca2016multistage,kammammettu2019two} and DRO \cite{gourtani2020distributionally,jin2022wasserstein,zhou2020linear} problems. Specifically, we use the affine function 
\[
x^{\rm a}_2\left({\xi}\right):=A{\xi} + a
\]
as our decision rule for $x_2$, where $A\in{\mathbb R}^{n_2\times m}$ and $a\in{\mathbb R}^{n_2}$ are determined simultaneously with $x_1$ at the first stage. Thus, the 2-DRLP of our interest is formulated as 
\begin{equation}\label{eq:appx}
\min_{x_1\in{\mathcal X}_1,\left(A,a\right)\in{\mathcal A}\left(x_1,{\Xi}\right)}c^\top_1 x_1 + h_\Xi\left(A,a\right)
\end{equation}
where
\begin{equation}\label{eq:h}
h_\Xi\left(A,a\right):=\max_{{\mathbb P}\in{\mathcal P}_{\varepsilon}\left({\Xi}\right)}{\mathbb E}_{{\mathbb P}}\left[c^\top_2\left(A\xi+a\right)\right]
\end{equation}
denotes the worst-case expected cost of wait-and-see decisions using the affine policy over ${\mathcal P}_{\varepsilon}\left(\Xi\right)$. To guarantee that $x^{\rm a}_2$ is feasible over $\Xi$, we define 
\[
\begin{aligned}
&{\mathcal A}\left(x_1,\Xi\right):=\left\{\left(A,a\right)\in{\mathbb R}^{n_2\times m}\times{\mathbb R}^{n_2}: A^{\rm in}_1 x_1 + A^{\rm in}_2\left(A{\xi} + a\right) + A^{\rm in}_3{\xi} \leq b^{\rm in}\quad \forall {\xi}\in\Xi\right\}.
\end{aligned}
\]
In this study, we assume that (\ref{eq:appx}) is feasible.\footnote{Unless $m=1$, however, (\ref{eq:appx}) might be infeasible even when (\ref{eq:problem}) is feasible \cite{bertsimas2012power,chen2008linear}. 
In this case, the following discussions throughout the article do not apply.}

One  reason (\ref{eq:problem}) is hard to solve in practice is its scalability issue regarding sample size. Intractable in the current form due to the nested infinite-dimensional optimization problem, (\ref{eq:problem}) can be rewritten in a tractable form using well-studied Wasserstein DRO techniques. However, the scale of any tractable reformulation of (\ref{eq:problem}) grows with sample size. We present such a tractable reformulation in the following proposition, which can be proven by duality theory; see, e.g., \cite{esfahani2018data,zhao2018data}. 
\begin{proposition}\label{prop:exa_reform}
Problem (\ref{eq:problem}) can be rewritten as the two-stage RO problem
\begin{equation}\label{eq:exa_reform}
\begin{aligned}
\min_{x_1\in{\mathcal X}_1, \lambda\geq0, \eta\in{\mathbb R}^N} \quad &c^\top_1 x_1+\lambda\varepsilon + \frac{1}{N}\sum_{i\in{\mathcal I}}\left[\eta\right]_i\\
\text{s.t.} \quad &f\left(x_1,{\xi}\right) - \lambda\onenorm{\xi-{\xi}_i}\leq\left[\eta\right]_i\quad\forall {\xi}\in{\Xi}, i\in{\mathcal I}. 
\end{aligned}
\end{equation}
\end{proposition}
Problem (\ref{eq:exa_reform}) can be solved using decomposition algorithms such as Benders decomposition, the C\&CG algorithm and variants of these  methods \cite{gamboa2021decomposition}. 
In these algorithms, (\ref{eq:exa_reform}) is decomposed into a master problem and two types of subproblems that are iteratively solved. Each of the master problem and subproblems is written as an MILP. The scalability issue regarding sample size is problematic specifically for the following two reasons. First, one of the two subproblems, which has a size independent of sample size, has to be solved for each sample at each iteration. 
Second, a set of decision variables and/or constraints, the number of which is proportional to sample size, can be added to the master problem at each iteration. As empirically shown  in \cite{till2003empirical}, this may well cause the actual computation time of the master problem to increase superlinearly with sample size. Moreover, undoubtedly, the master problem with a scale increasing with sample size makes a decomposition algorithm for (\ref{eq:exa_reform}) susceptible to memory-outage errors when many samples are used.

Considering the superior tractability of affine policies, one natural question arises: {\it Does (\ref{eq:appx}) suffer from the same scalability issue regarding sample size as (\ref{eq:problem})}? In the following section, we show that the answer is no, i.e., (\ref{eq:appx}) has a tractable reformulation with a scale independent of sample size. 

\begin{remark}\label{rem:ff}
The feasibility of (\ref{eq:problem}) implies that any feasible point $x_1$ should be such that ${\mathcal X}_2\left(x_1,{\xi}\right)$ is non-empty for any ${\xi}\in{\Xi}$, i.e.,  
\begin{equation}\label{eq:ff}
f^{\rm f}\left(x_1,{\xi}\right)=0,\quad\forall{\xi}\in{\Xi}
\end{equation}
where $f^{\rm f}\left(x_1,{\xi}\right)$ is equal to the optimal value of the LP 
\begin{equation}\label{eq:ffproblem} 
\begin{aligned}
\min_{x_2\in{\mathbb R}^{n_2}, y\in{\mathbb R}_+}\quad &y \\
\text{s.t.}\quad &A^{\rm in}_1x_1 + A^{\rm in}_2x_2 + A^{\rm in}_3{\xi} \leq b^{\rm in} + I_k y.
\end{aligned}
\end{equation}
In words, $f^{\rm f}\left(x_1,{\xi}\right)$ denotes the maximum violation of constraints in (\ref{eq:f}). By taking the dual formulation of (\ref{eq:ffproblem}), we observe that $f^{\rm f}\left(x_1,{\xi}\right)$ is convex in $\xi$ for a fixed $x_1$. Thus, (\ref{eq:ff}) is rewritten as 
\begin{equation}\label{eq:ff2} 
f^{\rm f}\left(x_1,{\xi}\right)=0,\quad\forall{\xi}\in{\mathcal V}\left(\Xi\right).
\end{equation}
We make explicit use of (\ref{eq:ff2}) in Section \ref{sec:support}, where we construct a Wasserstein ball different from ${\mathcal P}_{\varepsilon}\left(\Xi\right)$ and (\ref{eq:ff}) may not be implied. 
\end{remark}

\begin{remark}
Problem (\ref{eq:appx}) can also express a ``multi-stage" DRLP over 1-Wasserstein balls using an affine policy. Specifically, we consider the multi-stage DRLP 
\begin{equation}\label{eq:multi}
\begin{aligned}
&\min_{x_1\in{\mathcal X}_1} c^\top_1x_1
+
\max_{{\mathbb P}^2\in{\mathcal P}^2_\varepsilon\left(\Xi^2\right)}
{\mathbb E}_{{\mathbb P}^2}
\bigg[\min_{z_2\in{\mathcal Z}_2\left(x_1,\xi^2\right)}
e^\top_2 z_2 +\max_{{\mathbb P}^3\in{\mathcal P}^3_\varepsilon\left(\Xi^3\right)}{\mathbb E}_{\mathbb P^3}\bigg[\min_{z_3\in{\mathcal Z}_3\left(x_1,z_2,\xi^2,\xi^3\right)} e^\top_3 z_3 + \cdots \\
&\ \ +\max_{{\mathbb P}^T\in{\mathcal P}^T_\varepsilon\left(\Xi^T\right)}
{\mathbb E}_{\mathbb P^T}\bigg[\min_{z_T\in{\mathcal Z}_T\left(x_1,z_2,\ldots,z_{T-1},\xi^2,\ldots,\xi^T\right)} e^\top_T z_T
\bigg]\bigg]
\bigg]
\end{aligned}
\end{equation}
where $\xi^t$, $\Xi^t$ and ${\mathcal P}^t_\varepsilon\left(\Xi^t\right)$ denote a random vector, its rectangular support, and a 1-Wasserstein ball for each stage $t=2,\ldots,T$, respectively. Furthermore, $z_t$, ${\mathcal Z}_t$, and $e_t$ denote a real decision vector, its feasible set defined using a finite number of linear inequalities with right-hand-side uncertainty, and its cost coefficient vector for each stage $t$, respectively. Using the affine function 
\[
z^{\rm a}_t\left(\xi^2,\ldots,\xi^t\right) := \sum_{\tau=2}^t A^\tau \xi^\tau + a^\tau
\]
as a decision rule for $z_t$ in (\ref{eq:multi}), which depends on the realization of uncertainty only up to stage $t$, we can formulate a multi-stage problem in the form of (\ref{eq:appx}) for $\xi=\left(\xi^2,\ldots,\xi^T\right)$. Here, the matrices $A^\tau$ and vectors $a^\tau$ to be determined at the first stage are of appropriate dimensions. However, it is unclear if affine policies for multi-stage DRLPs over Wasserstein balls with different assumptions and problem structures, possibly of greater practical importance, lead to (\ref{eq:appx}) in a similar way. Thus, we focus on the two-stage formulation (\ref{eq:problem}) in this article. For details on general multi-stage DRO or distributionally robust dynamic programming problems, the reader is referred to, for example, \cite{Duque2020,yang2020wasserstein,zhang2022distributionally,rahimian2022effective}. 
\end{remark}

\section{Independence of Sample Size}\label{sec:appx}
Similar to (\ref{eq:problem}), (\ref{eq:appx}) is intractable in the current form as (\ref{eq:h}) is infinite-dimensional. In this section, we prove that (\ref{eq:appx}) has a tractable reformulation with a scale independent of sample size. In particular, we derive a finite-dimensional MILP equivalent to (\ref{eq:appx}), the scale of which is invariant with sample size. Subsequently, we present a cutting-plane algorithm for solving the reformulated problem. To this end, we first prove the following theorem. 
\begin{theorem}
Problem (\ref{eq:h}) is rewritten as an LP with a scale independent of $N$. Specifically, we have 
\begin{align}
h_\Xi\left(A,a\right)=
\max_{\tilde{q}^+,\tilde{q}^-\in{\mathbb R}^m_+} \quad 
&c^\top _2\left\{A\left(\tilde{\xi} + \tilde{q}^+ - \tilde{q}^-\right)+a\right\}\nonumber\\
\text{s.t.} \quad &1^\top_m\left(\tilde{q}^++\tilde{q}^-\right)\leq \varepsilon\label{eq:h_con1}\\
& \tilde{q}^+ \leq \overline{\xi} - \tilde{\xi}\label{eq:h_con2}\\
&  \tilde{q}^- \leq \tilde{\xi} - \underline{\xi}\label{eq:h_con3}
\end{align}
where $\tilde{\xi}:=\frac{1}{N}\sum_{i\in{\mathcal I}} {\xi}_i$. 
\end{theorem}
\begin{proof}
Let 
\begin{equation*}
\overline{h}_\Xi\left(A\right):=
\max_{{\mathbb P}\in{\mathcal P}_{\varepsilon}\left({\Xi}\right)}{\mathbb E}_{{\mathbb P}}
\left[c^\top_2{A{\xi}}\right].
\end{equation*}
From Theorem 4.4 in \cite{esfahani2018data}, it follows that 
\begin{align}
\overline{h}_\Xi\left(A\right)=
\max_{q\in{\mathbb R}^{Nm}} \quad &\frac{1}{N}\sum_{i\in{\mathcal I}} c^\top _2A\left({\xi}_i + q_i\right)\nonumber\\
\text{s.t.} \quad &\frac{1}{N}\sum_{i\in{\mathcal I}} \onenorm{q_{i}}\leq \varepsilon\label{eq:normconstraint}\\
& \underline{\xi}\leq {\xi}_i + q_i\leq \overline{\xi} \quad\forall i\in{\mathcal I}\nonumber
\end{align}
where $q:=\left(q_1,\ldots,q_N\right)$ is a vector concatenating $q_i\in{\mathbb R}^m_+$ for all $i\in{\mathcal I}$. Introducing auxiliary decision vectors $q^{+}_i, q^-_i\in{\mathbb R}^m_+$ such that $q_i=q^{+}_i - q^-_i$ and $q^+_i\circ q^-_i=0_m$ for each $i\in{\mathcal I}$ to linearize the norm constraint (\ref{eq:normconstraint}), we observe that 
\begin{align}
\overline{h}_\Xi\left(A\right)=
 \max_{q^+,q^-\in{\mathbb R}^{Nm}_+}
\quad &\frac{1}{N}\sum_{i\in{\mathcal I}} c^\top _2A\left({\xi}_i + q^+_i - q^-_i\right)\nonumber\\
\text{s.t.} \quad &\frac{1}{N}\sum_{i\in{\mathcal I}}1^\top_m\left(q^+_{i}+q^-_i\right)\leq \varepsilon \label{eq:normconstraintlinear}\\
& q^+_i \leq \overline{\xi} - {\xi}_i \quad\forall i\in{\mathcal I}\label{eq:ineqs_plus}\\
& q^-_i\leq {\xi}_i - \underline{\xi}\quad\forall i\in{\mathcal I}\label{eq:ineqs_minus}\\
& q^+_i\circ{q}^-_i={0}_m\quad \forall i\in{\mathcal I}\label{eq:mutual}
\end{align}
where $q^+:=\left(q^+_1,\ldots,q^+_N\right)$ and ${q}^-:=\left({q}^-_1,\ldots,{q}^-_N\right)$. Note that the mutual exclusivity constraint (\ref{eq:mutual}) is redundant and thus can be omitted without affecting optimality. Adding up the $N$ inequalities in (\ref{eq:ineqs_plus}) and those in (\ref{eq:ineqs_minus}) respectively, we further have 
\begin{align}
\overline{h}_\Xi\left(A\right)\leq
\max_{{q}^+,{q}^-\in{\mathbb R}^{Nm}_+} \quad & \frac{1}{N}\sum_{i\in{\mathcal I}} c^\top _2A\left({\xi}_i + q^+_i - {q}^-_i\right)\nonumber\\
\text{s.t.}\quad &\text{(\ref{eq:normconstraintlinear})}\nonumber\\
& \sum_{i\in{\mathcal I}}{q}^+_i\leq N\overline{\xi} - \sum_{i\in{\mathcal I}}{\xi}_i \label{eq:ineqs_plus_add}\\
& \sum_{i\in{\mathcal I}}{q}^-_i \leq \sum_{i\in{\mathcal I}}{\xi}_i-N\underline{\xi}. \label{eq:ineqs_minus_add}
\end{align}
In what follows, we show that this holds as equality. For any ${q}^{+\prime}=\left({q}^{+\prime}_1,\ldots,{q}^{+\prime}_N\right)\in\left\{{q}^+\in{\mathbb R}^{Nm}_+:\text{(\ref{eq:ineqs_plus_add})}\right\}$, there exists ${q}^{+\prime\prime}=\left({q}^{+\prime\prime}_1,\ldots,{q}^{+\prime\prime}_N\right)\in\left\{{q}^+\in{\mathbb R}^{Nm}_+:\text{(\ref{eq:ineqs_plus})}\right\}$ such that $\sum_{i\in{\mathcal I}}{q}^{+\prime\prime}_i = \sum_{i\in{\mathcal I}}{q}^{+\prime}_i$. For example, one such ${q}^{+\prime\prime}$ can be obtained by letting 
\[
{q}^{+\prime\prime}_{i} :=
\begin{cases}
\begin{aligned}
&\min\left\{\overline{\xi} - {\xi}_i,\sum_{j\in{\mathcal I}} {q}^{+\prime}_j \right\} && i=1\\
&\min\left\{\overline{\xi} - {\xi}_{i},\sum_{i\in{\mathcal I}} {q}^{+\prime}_i - \sum_{j<i}{q}^{+\prime}_{j}\right\}&&i\geq2,
\end{aligned}
\end{cases}
\]
where the minimum is taken entrywisely. For any $q^{-\prime}=\left({q}^{-\prime}_1,\ldots,{q}^{-\prime}_N\right)\in\left\{{q}^+\in{\mathbb R}^{Nm}_+:\text{(\ref{eq:ineqs_minus_add})}\right\}$, similarly, we have at least one ${q}^{+\prime\prime}=\left({q}^{-\prime\prime}_1,\ldots,{q}^{-\prime\prime}_N\right)\in\left\{{q}^-\in{\mathbb R}^{Nm}_+:\text{(\ref{eq:ineqs_minus})}\right\}$ such that $\sum_{i\in{\mathcal I}}{q}^{-\prime\prime}_i = \sum_{i\in{\mathcal I}}{q}^{-\prime}_i$. Thus, it holds that 
\begin{align}
\overline{h}_\Xi\left(A\right)=
\max_{{q}^+,{q}^-\in{\mathbb R}^{Nm}_+} \quad &\frac{1}{N}\sum_{i\in{\mathcal I}} {c}^\top _2A\left({\xi}_i + {q}^+_i - {q}^-_i\right)\nonumber\\
\text{s.t.}\quad &\text{(\ref{eq:normconstraintlinear}), (\ref{eq:ineqs_plus_add}), (\ref{eq:ineqs_minus_add}). }\nonumber
\end{align}
By adding $c^\top_2{a}$ to both sides and letting $\tilde{q}^+:=\sum_{i\in{\mathcal I}}{q}^+_i$ and $\tilde{q}^-:=\sum_{i\in{\mathcal I}}{q}^-_i$, we prove the statement. 
\qed
\end{proof}
Based on duality in LPs, we have 
\[
\begin{aligned}
h_\Xi\left({A},a\right)=\min_{\mu\in{\mathcal M}\left(A\right)} {c}_{3,\Xi}^\top {\mu}+c_2^\top\left(A\tilde{\xi}+a\right)
\end{aligned}
\]
where ${\mu}:=\left(\mu^0,\mu^+,\mu^-\right)\in{\mathbb R}^{2m+1}_+$ with $\mu^0\in{\mathbb R}_+$, $\mu^+\in{\mathbb R}^{m}_+$, and $\mu^-\in{\mathbb R}^{m}_+$ denote the dual decision variable and vectors associated with constraints (\ref{eq:h_con1})--(\ref{eq:h_con3}), respectively, ${c}_{3,\Xi}:=(\varepsilon,\overline{\xi}-\tilde{\xi}, \tilde{\xi}-\underline{\xi})\in{\mathbb R}^{2m+1}$, and 
\[
\begin{aligned}
&{\mathcal M}\left(A\right):=\left\{{\mu}\in{\mathbb R}^{2m+1}_+:
\begin{bmatrix}
1_m & I_m & O_m \\
1_m & O_m & I_m
\end{bmatrix}
{\mu}\geq
\begin{bmatrix}
A^\top c_2\\
-A^\top c_2
\end{bmatrix}
\right\}.
\end{aligned}
\]
Thus, (\ref{eq:appx}) is rewritten as the semi-infinite MILP 
\begin{equation}\label{eq:appx_re1}
\min_{\substack{{x}_1\in{\mathcal X}_1,
\left(A,{a}\right)\in{\mathcal A}\left({x}_1,{\Xi}\right)\\
{\mu}\in{\mathcal M}\left(A\right)}}
{c}^\top_1{x}_1 + {c}^\top_2\left(A\tilde{\xi}+{a}\right)+{c}^\top_{3,\Xi}{\mu}.
\end{equation}
Problem (\ref{eq:appx_re1}) is semi-infinite as ${\mathcal A}\left({x}_1,\Xi\right)$ is defined with an infinite number of inequalities. Since the inequalities are linear in ${\xi}$ for a fixed $\left(A,{a}\right)$, it can be replaced with 
\[
\begin{aligned}
&{\mathcal A}^{\rm v}\left({x}_1,\Xi\right):=
\left\{\left(A,{a}\right)\in{\mathbb R}^{n_2\times m}\times{\mathbb R}^{n_2}: A^{\rm in}_1{x}_1 + A^{\rm in}_2\left(A{\xi} + {a}\right) + A^{\rm in}_{3}{\xi} \leq {b}^{\rm in}\quad \forall {\xi}\in{\mathcal V}\left(\Xi\right)\right\}.
\end{aligned}
\]
Thus, (\ref{eq:appx_re1}) is rewritten as the finite-dimensional MILP
\begin{equation}\label{eq:appx_re2}
\min_{\substack{{x}_1\in{\mathcal X}_1,\left(A,{a}\right)\in{\mathcal A}^{\rm v}\left({x}_1,{\Xi}\right),\\
{\mu}\in{\mathcal M}\left(A\right)}}{c}^\top_1{x}_1 + {c}^\top_2\left(A\tilde{\xi}+{a}\right)+{c}^\top_{3,\Xi}{\mu}.
\end{equation}
However, (\ref{eq:appx_re2}) is still hard to handle using off-the-shelf MILP solvers due to the large number $\lvert{\mathcal V}\left(\Xi\right)\rvert L=2^mL$ of linear inequalities defining ${\mathcal A}^{\rm v}\left({x}_1,{\Xi}\right)$, which may incur timeout or memory-outage errors. To avoid these errors, we solve (\ref{eq:appx_re2}) using a cutting-plane algorithm, assuming that an off-the-shelf MILP solver is available.

The algorithm for (\ref{eq:appx_re2}) is described as follows. For initialization, we select any ${\xi} _{l1}\in{\mathcal V}\left(\Xi\right)$ and let ${\Xi}^{\rm v}_{l1}:=\left\{{\xi}_{l1}\right\}$ for each $l\in{\mathcal L}:=\left\{1,\ldots,L\right\}$. At each iteration $P\geq 1$, we solve the master problem
\begin{equation}\label{eq:master}
\min_{\substack{{x}_1\in{\mathcal X}_1,
\left(A,{a}\right)\in{\mathcal A}^{\rm v}_P\left({x}_1,{\Xi}^{\rm v}_P\right),\\
{\mu}\in{\mathcal M}\left(A\right)
}}
{c}^\top_1{x}_1 + {c}^\top_2\left(A\tilde{\xi}+{a}\right) + {c}^\top_{3,\Xi}{\mu}
\end{equation}
where $\Xi^{\rm v}_{P}:=\left(\Xi^{\rm v}_{1P},\ldots,\Xi^{\rm v}_{LP}\right)$ and
\[
\begin{aligned}
&{\mathcal A}^{\rm v}_P\left(x_1,\Xi^{\rm v}_P\right):= \{\left(A,a\right)\in{\mathbb R}^{n_2\times m}\times{\mathbb R}^{n_2}: 
\left.\left[A^{\rm in}_1x_1 + A^{\rm in}_2\left(A{\xi} + a\right)+A^{\rm in}_3{\xi}\right]_l \leq \left[b^{\rm in}\right]_l \forall {\xi}\in\Xi^{\rm v}_{lP},l\in{\mathcal L}\right\}.
\end{aligned}
\]
Problem (\ref{eq:master}) is an MILP. Let $\left(x_{1P},A_P,a_P\right)$ and $L_P$ denote the solution corresponding to $\left(x_{1},A,a\right)$ and the optimal value of (\ref{eq:master}), respectively. Subsequently, for each $l\in{\mathcal L}$, we solve the subproblem 
\begin{equation}\label{eq:sub}
\max_{{\xi}\in{\Xi}}\left[A^{\rm in}_1{x}_1 + A^{\rm in}_2\left(A_P{\xi} + {a}_P\right) + A^{\rm in}_3{\xi} - {b}^{\rm in}\right]_l
\end{equation}
which is an LP. Let ${\xi}_{lP}$ and $F_{lP}$ denote the solution and optimal value of (\ref{eq:sub}), respectively. We assume that (\ref{eq:sub}) is solved by a simplex method such that ${\xi}_{lP}\in{\mathcal V}\left(\Xi\right)$. If $F_{lP}$ is greater than a pre-defined feasibility tolerance $\rho\geq0$, it is implied that the constraint 
\[
\left[A^{\rm in}_1{x}_1 + A^{\rm in}_2\left(A{\xi}_{lP} + {a}\right) + A^{\rm in}_3{\xi}_{lP}\right]_l\leq\left[{b}^{\rm in}\right]_l
\]
in (\ref{eq:appx_re2}) is violated. We let ${\Xi}^{\rm v}_{l(P+1)}:={\Xi}^{\rm v}_{lP}\cup\left\{{\xi}_{lP}\right\}$ in this case and ${\Xi}^{\rm v}_{l(P+1)}:={\Xi}^{\rm v}_{lP}$ otherwise. If $F_P:=\max_{l\in{\mathcal L}}F_{lP}$ is no greater than $\rho$, it is implied that all the constraints in (\ref{eq:appx_re2}) are met. Thus, the algorithm stops and $\left({x}^\ast_{1},A^\ast,{a}^\ast\right)=\left({x}_{1P},A_{P},{a}_{P}\right)$ is returned as a solution to (\ref{eq:appx}). Otherwise, the iteration step increases and we solve (\ref{eq:master}) again. Problem (\ref{eq:master}) is a relaxation of (\ref{eq:appx}) for any iteration step $P$ such that $F_P>0$. Thus, $L_P$ monotonically converges to the optimal value of (\ref{eq:appx}). Moreover, as $\lvert{\mathcal V}\left(\Xi\right)\rvert<\infty$, the algorithm yields a solution optimal within the optimality tolerance of the off-the-shelf MILP solver in finitely many iterations. We provide a pseudocode of the algorithm in Algorithm \ref{alg:appx_re2}. 

\begin{algorithm}[t!]
\caption{Algorithm for (\ref{eq:appx})}\label{alg:appx_re2}
\begin{algorithmic}
\Require Feasibility tolerance $\rho\geq0$, any ${\xi}_{l1}\in{\mathcal V}\left(\Xi\right)$ for each $l\in{\mathcal L}$
\Ensure Solution $\left({x}_1^\ast,A^{\ast},{a}^{\ast}\right)$ to (\ref{eq:appx})
\For{$l\gets 1$ to $L$}
\State{$\Xi^{\rm v}_{l1}\gets\left\{{\xi}_{l1}\right\}$}
\EndFor
\State{$F_1\gets \infty$, $P\gets 1$}
\While{$F_P>\rho$}
\State{Solve (\ref{eq:master}), 
$\left({x}_1^\ast,A^{\ast},{a}^{\ast}\right)\gets
\left({x}_{1P},A_P,{a}_P\right)$}
\For{$l\gets 1$ to $L$}
\State{Solve (\ref{eq:sub})}
\If{$F_{lP}>\rho$}
\State{$\Xi^{\rm v}_{l(P+1)}\gets\Xi^{\rm v}_{lP}\cup\left\{{\xi}_{lP}\right\}$}
\Else
\State{$\Xi^{\rm v}_{l(P+1)}\gets\Xi^{\rm v}_{lP}$}
\EndIf
\EndFor
\State{$F_{P+1}\gets\max_l F_{lP}$, $P\gets P+1$}
\EndWhile
\end{algorithmic}
\end{algorithm}

Note that we do not actually use the optimized affine policy $x^{\rm a}_2\left(\xi\right)=A^\ast{\xi}+a^\ast$ in any decision-making stage. Rather, we enjoy only the computational tractability of affine policies when determining  $x_1$ in the first stage. In the second stage, we do not rely on the affine policy to determine ${x}_2$ as it may be overly conservative. Instead, we solve (\ref{eq:f}) for ${x}_1={x}^\ast_1$ to determine ${x}_2$, the feasibility of which for any $\xi\in\Xi$ is implied by the feasibility of (\ref{eq:appx}). As (\ref{eq:f}) is a standard LP, we can always make more efficient wait-and-see decisions compared to using the affine policy. 

According to \cite{bertsimas2012power}, the optimality gap of (\ref{eq:problem}) and (\ref{eq:appx}) incurred by the affine policy can be arbitrarily large, when $\varepsilon$ is big enough so that (\ref{eq:problem}) is identical to its RO counterpart. This can also be true for any $\varepsilon>0$ as discussed in what follows. We first present the following theorem. 
\begin{theorem}\label{theorem:nondecreasing}
The optimal values of (\ref{eq:problem}) and (\ref{eq:appx}) as a function of $\varepsilon>0$ are piecewise affine and concave. 
\end{theorem}
\begin{proof}
It is enough to address only the optimal value of (\ref{eq:problem}). For any $x_1\in{\mathcal X}_1$ such that (\ref{eq:ff2}) holds, $\lambda\geq0$, and $i\in{\mathcal I}$, we consider the problem
\begin{equation}\label{eq:lemma}
\max_{\xi\in\Xi}f\left(x_1,{\xi}\right) - \lambda\onenorm{\xi-{\xi}_i}.
\end{equation}
Introducing decision vectors $r^+_i,r^-_i\in{\mathbb R}^m_+$ such that $\xi = \xi_i + r^+_i - r^-_i$, (\ref{eq:lemma}) is rewritten as 
\begin{equation}\label{eq:lemmaproof}
\begin{aligned}
&\max_{\substack{(r^+_i,r^-_i)\in{\mathcal R}\left(\sigma\right),\sigma\in\left\{0,1\right\}^m}}
f\left(x_1,{\xi_i+r^+_i-r^-_i}\right) - \lambda\left(r^+_i+r^-_i\right)
\end{aligned}
\end{equation}
where 
\[
\begin{aligned}
&{\mathcal R}\left(\sigma\right):=\left\{\left(r^+_i,r^-_i\right)
\in{\mathbb R}^m_+\times{\mathbb R}^m_+: r^+_i\leq \left(\overline{\xi}-\xi_i\right)\circ \sigma, \ r^-_i\leq \left(\xi_i - \underline{\xi}\right)\circ\left(1_m-\sigma\right)
\right\}.
\end{aligned}
\]
As the dual of (\ref{eq:f}) for $\xi=\xi_i + r^+_i - r^-_i$ is a maximization problem with an objective function that is linear in $(r^+_i,r^-_i)$ for any fixed dual vector, the feasible set ${\mathcal R}\left(\sigma\right)$ of $(r^+_i,r^-_i)$ in (\ref{eq:lemmaproof}) can be replaced with its vertex set. This implies that a solution $\xi$ to (\ref{eq:lemma}) can be assumed to be equal to $\xi_i$, $\overline{\xi}$, or $\underline{\xi}$ in each entry independently. 
Thus, (\ref{eq:lemma}) is rewritten as the integer program
\[
\begin{aligned}
\max_{\sigma^+_{i},\sigma^-_{i}\in\left\{0,1\right\}^{m}}\quad &f\left(x_1,{\xi}\right) - \lambda\onenorm{\xi-{\xi}_i}\\
\text{s.t.} \quad &\xi = \xi_i + \left(\overline{\xi}-\xi_i\right)\circ\sigma^+_i - \left(\xi_i-\underline{\xi}\right)\circ\sigma^-_i\\
&\sigma^+_i+\sigma^-_i\leq 1_m.
\end{aligned}
\]
Based on this equivalence, (\ref{eq:exa_reform}) can be rewritten as a finite-dimensional MILP by introducing wait-and-see decision vectors associated with uncertain scenarios corresponding to each pair of $\sigma^+_i,\sigma^-_i\in\{0,1\}^{m}$ such that $\sigma^+_i+\sigma^-_i\leq1_m$ for each $i\in{\mathcal I}$. Thus, the optimal value of (\ref{eq:problem}) can be considered as a point-wise minimum of finitely many affine functions of $\varepsilon$. Hence the statement holds. 
\qed
\end{proof}

Theorem \ref{theorem:nondecreasing} suggests that the optimality gap between (\ref{eq:problem}) and (\ref{eq:appx}) as a function of $\varepsilon>0$ is a difference of two concave functions, which can be neither increasing nor decreasing in general. As there exists some $\varepsilon$ such that the optimality gap can be arbitrarily large, we assert that this holds for any $\varepsilon>0$. Although conditions under which affine policies for two-stage Wasserstein DRO problems can be optimal are studied in \cite{georghiou2021optimality}, we do not have such special assumptions on (\ref{eq:problem}) as most real-world problems do not satisfy them. To reduce the inevitable conservativeness of the affine policy, in the following section, we build an uncertainty set smaller than $\Xi$ and re-define the Wasserstein ball on it. 

\section{Conservativeness Reduction via Wasserstein Ball Refinement}\label{sec:support}
To reduce the conservativeness of the affine policy, we propose to use a Wasserstein ball ${\mathcal P}_\varepsilon\left(\Omega\right)$ instead of ${\mathcal P}_\varepsilon\left(\Xi\right)$, which is defined on a data-driven uncertainty set $\Omega:={\Xi}\cap{\Xi}^{\rm a}$. Here, we define 
\[
\Xi^{\rm a} := [{\xi}^{\rm l}-\varepsilon\Delta{1}_m,{\xi}^{\rm u}+\varepsilon\Delta{1}_m]
\] 
where ${\xi}^{\rm l}\in{\mathbb R}^m$ and ${\xi}^{\rm u}\in{\mathbb R}^m$ are the entry-wise minimum and maximum vectors of the samples, respectively, i.e., $[{\xi}^{\rm l},{\xi}^{\rm u}]$ is the box hull of the samples. We let $\Delta := \max\left\{N,\beta\right\}$ where $\beta>0$ is a user-defined parameter. 

Built this way, the uncertainty set $\Omega$ is endowed with a probabilistic property stated in the following theorem.
\begin{theorem}\label{theorem1}
The worst-case probability of the realization of $\xi$ being outside $\Omega$ over $\mathcal P_\varepsilon\left({\Xi}\right)$ is bounded by $\Delta^{-1}$, i.e., 
\[
\sup_{{\mathbb P}\in{\mathcal P}_{\varepsilon}\left({\Xi}\right)}{\mathbb P}\left[{\xi}\notin\Omega\right]\leq \Delta^{-1}.
\]
\end{theorem}
\begin{proof}
Assume for the proof that $\xi$ has a compact convex support ${\Xi}^{\rm u}$ whose interior contains ${\Xi}\cup{\Xi}^{\rm a}$. According to Theorem 4.4 and Corollary 5.3 in \cite{esfahani2018data},
\[
V := \sup_{{\mathbb P}\in{\mathcal P}_\varepsilon\left({\Xi}^{\rm u}\right)}{\mathbb P}\left[{\xi}\notin\left(\Xi^{\rm a}\right)^\circ\right]
\]
is equal to the optimal value of the problem
\begin{equation}\label{eq:theorembeta}
\begin{aligned}
\sup_{\substack{\alpha\in{\mathbb R}^{N(2N+1)}_+,\\{p}\in{\mathbb R}^{mN(2N+1)}}} \quad &\frac{1}{N}\sum_{i\in{\mathcal I}}\sum_{k\in{\mathcal K}}\alpha_{ik}l_{k}\left({\xi}_i + \frac{{p}_{ik}}{{\alpha_{ik}}}\right)\\
\text{s.t.} \quad &\frac{1}{N}\sum_{i\in{\mathcal I}}\sum_{k\in{\mathcal K}}\onenorm{{p}_{ik}}\leq \varepsilon,\\
&\sum_{k\in{\mathcal K}}\alpha_{ik} = 1\quad \forall i\in{\mathcal I},\\
&{\xi}_i + 
\frac{{p}_{ik}}{{\alpha_{ik}}} \in {\Xi}^{\rm u}\quad \forall k\in{\mathcal K}, i\in{\mathcal I}
\end{aligned}
\end{equation}
where ${\alpha}$ and ${p}$ are a vector of $\alpha_{ik}\in{\mathbb R}_+$ and a vector concatenating ${p}_{ik}\in{\mathbb R}^m$ for all $\left(i,k\right)\in{\mathcal I}\times{\mathcal K}$ with ${\mathcal K}:=\left\{1,2,\ldots,2N+1\right\}$, respectively. Moreover, we define 
\[
\begin{aligned}
&l_k\left({\xi}\right):=
\begin{cases}
\begin{aligned}
&1 && \text{if}\quad 
\left[{\xi}\right]_k \geq
\left[{\xi}_{\rm u}+\varepsilon\Delta\right]_k\\
&-\infty && \text{otherwise}
\end{aligned}
\end{cases} \forall k\in{\mathcal I},\\
&l_{N+k}\left({\xi}\right):=
\begin{cases}
\begin{aligned}
&1 && \text{if}\quad \left[\xi\right]_k \leq \left[{\xi}_{\rm l}-\varepsilon\Delta\right]_k\\
&-\infty && \text{otherwise}
\end{aligned}
\end{cases} \forall k\in{\mathcal I},
\end{aligned}
\]
and $l_{2N+1}\left({\xi}\right):=0$. In (\ref{eq:theorembeta}), the conventional extended arithmetics apply. For example, we have $1/0 = \infty$, $0/0 = 0$, and $0\cdot\infty = 0$. The optimal value of (\ref{eq:theorembeta}) is obtained if either $\left[{p}_{ik}\right]_k=\varepsilon N$ for any $\left(i,k\right)\in{\mathcal I}\times{\mathcal I}$ such that $i=\argmaxA_{i^\prime\in{\mathcal I}}\left[\xi_{i^\prime}\right]_{k}$ or $\left[{p}_{ik}\right]_{k-N}=-\varepsilon N$ for any $\left(i,k\right)\in{\mathcal I}\times\left\{N+1,\ldots,2N\right\}$ such that $i=\argminA_{i^\prime\in{\mathcal I}}\left[\xi_{i^\prime}\right]_{k}$, with $\alpha_{ik}=\max\left\{1,N/\beta\right\}$ in either case. These cases are when a sample originally closest to the boundary of $\Xi^{\rm a}$ moves along the shortest path to reach it by $\varepsilon\Delta$. Thus, we have $V= \Delta^{-1}$. Further, we observe that 
\[
\begin{aligned}
V 
&\geq \sup_{{\mathbb P}\in{\mathcal P}_\varepsilon\left({\Xi}^{\rm u}\right)}{\mathbb P}\left[{\xi}\notin
\Xi^{\rm a}\right]\\
&\geq \sup_{{\mathbb P}\in{\mathcal P}_\varepsilon\left({\Xi}^{\rm u}\right)\cap
\left\{{\mathbb P}^\prime\in{\mathcal P}\left({\Xi}^{\rm u}\right): {\mathbb P}^\prime\left({\xi}\in{\Xi}\right)=1\right\}}{\mathbb P}\left[{\xi}\notin\Xi^{\rm a}\right]\\
&= \sup_{{\mathbb P}\in{\mathcal P}_\varepsilon\left({\Xi}^{\rm u}\right)\cap\left\{{\mathbb P}^\prime\in{\mathcal P}\left({\Xi}^{\rm u}\right): {\mathbb P}^\prime\left({\xi}\in{\Xi}\right)=1\right\}}{\mathbb P}\left[{\xi}\notin{\Omega}\right]\\
&= \sup_{{\mathbb P}\in{\mathcal P}_\varepsilon\left(\Xi\right)}{\mathbb P}\left[{\xi}\notin{\Omega}\right].
\end{aligned}
\]
Hence the statement holds. 
\qed
\end{proof}

Replacing the Wasserstein ball ${\mathcal P}_\varepsilon\left(\Xi\right)$ in (\ref{eq:appx}) with ${\mathcal P}_{\varepsilon}\left(\Omega\right)$, we obtain the problem
\begin{equation}\label{eq:appx2}
\min_{x_1\in{\mathcal X}_1,\left(A,a\right)\in{\mathcal A}\left(x_1,{\Omega}\right)}
c^\top_1x_1 + h_\Omega\left(A,a\right)
\quad\text{s.t.\quad(\ref{eq:ff2})}
\end{equation}
where we impose (\ref{eq:ff2}) because the second-stage problem should always be feasible, as stated in Remark \ref{rem:ff}. Since the feasibility of the affine policy is ensured over $\Omega\subseteq\Xi$, (\ref{eq:appx2}) can yield a less conservative solution than (\ref{eq:appx}).

One of the biggest advantages of constructing $\Omega$ in the above-described way  is that we can preserve the independence of sample size discussed in Section \ref{sec:appx}. If not considering this property, one might be able to obtain an even smaller uncertainty set with a probabilistic guarantee using existing methods, e.g., \cite{duan2018distributionally} and \cite{poolla2020wasserstein}, which are mostly approximations of Wasserstein distributionally robust chance constraints. However, the existing methods can yield an uncertainty set that does not include all the historical samples. As we cannot apply an affine policy for samples outside the uncertainty set, the scalability issue may still exist in this case. Therefore, we develop the new method for building $\Omega$, which includes all the samples. 

Problem (\ref{eq:appx2}) can be solved by combining Benders decomposition or the C\&CG algorithm for addressing (\ref{eq:ff2}), which consists of many equalities, with the cutting-plane algorithm for solving (\ref{eq:appx}). In this article, we choose the C\&CG algorithm, which is reportedly faster than Benders decomposition \cite{zeng2013solving}. In what follows, we explain the resulting iterative algorithm for (\ref{eq:appx2}). Some symbols used to describe the cutting-plane algorithm for (\ref{eq:appx}) in Section \ref{sec:appx} may be re-defined. 

Based on the discussions in the previous section, we first reformulate (\ref{eq:appx2}) as 
\begin{equation}\label{eq:appx3}
\begin{aligned}
\min_{\substack{x_1\in{\mathcal X}_1,{\mu}\in{\mathcal M}\left(A\right),\\
\left(A,a\right)\in{\mathcal A}^{\rm v}\left(x_1,{\Omega}\right)}}
\quad &c^\top_1x_1 + c^\top_2\left(A\tilde{\xi}+a\right)+c^\top_{3,\Omega}{\mu}\\
\text{s.t.}\quad &\text{(\ref{eq:ff2})}.
\end{aligned}
\end{equation}
Subsequently, by introducing a wait-and-see decision vector ${x}^{\rm f}_{2q}\in{\mathbb R}^{n_2}$ associated with the $q$th vertex ${\xi}^{\rm f}_{q}$ of $\Xi$ for each $q\in{\mathcal Q}:=\left\{1,\ldots,\lvert{\mathcal V}\left({\Xi}\right)\rvert\right\}$ to deal with (\ref{eq:ff2}), we reformulate (\ref{eq:appx3}) as the finite-dimensional MILP
\begin{equation}\label{eq:appx4}
\begin{aligned}
\min_{\substack{x_1\in{\mathcal X}_1,{\mu}\in{\mathcal M}\left(A\right), {x}^{\rm f}_{2q},\\
\left(A,{a}\right)\in{\mathcal A}^{\rm v}\left({x}_1,{\Omega}\right)}} \quad & {c}^\top_1{x}_1 + {c}^\top_2\left(A\tilde{\xi}+{a}\right) + {c}^\top_{3,\Omega}{\mu}\\
\text{s.t.}\quad & A^{\rm in}_1{x}_1 + A^{\rm in}_2{x}^{\rm f}_{2q} + A^{\rm in}_3{\xi}^{\rm f}_{q} \leq {b}^{\rm in}\quad\forall q\in{\mathcal Q}.
\end{aligned}
\end{equation}
Similar to (\ref{eq:appx}), we decompose the large-scale problem (\ref{eq:appx4}) into a master problem and subproblems that are iteratively solved. For initialization, we select any ${\xi}_{l1}\in{\mathcal V}\left(\Omega\right)$ and define $\Omega^{\rm v}_{l1}:=\left\{{\xi}_{l1}\right\}$ for each $l\in{\mathcal L}$. We also select any ${\xi}^{\rm f}_1\in{\mathcal V}\left(\Xi\right)$ and let ${\mathcal Q}_1:=\left\{Q_1\right\}$ with $Q_1=1$. At each iteration $P\geq1$, we solve the master problem
\begin{equation}\label{eq:sec4_master}
\begin{aligned}
\min_{\substack{{x}_1\in{\mathcal X}_1,{\mu}\in{\mathcal M}\left(A\right), {x}^{\rm f}_{2q},\\
\left(A,{a}\right)\in{\mathcal A}^{\rm v}_P\left({x}_1,{\Omega}^{\rm v}_P\right)}} \quad &
{c}^\top_1{x}_1 + {c}^\top_2\left(A\tilde{\xi}+{a}\right) + {c}^\top_{3,\Omega}{\mu}\\
\text{s.t.}\quad & A^{\rm in}_1{x}_1 + A^{\rm in}_2{x}^{\rm f}_{2q} + A^{\rm in}_3{\xi}^{\rm f}_{q} \leq {b}^{\rm in}\quad\forall q\in{\mathcal Q}_P
\end{aligned}
\end{equation}
where $\Omega^{\rm v}_P:=\left(\Omega^{\rm v}_{1P},\ldots,\Omega^{\rm v}_{LP}\right)$. 
Let $\left({x}_{1P},A_{P},{a}_P\right)$ and $L_P$ denote the solution corresponding to $\left({x}_1,A,{a}\right)$ and optimal value of (\ref{eq:sec4_master}), respectively. 
Subsequently, we solve the first subproblem
\begin{equation}\label{eq:maxff}
\max_{{\xi}\in{\mathcal V}\left({\Xi}\right)} f^{\rm f}\left({x}_{1P},{\xi}\right)
\end{equation}
whose solution and optimal value are denoted by ${\xi}^{\rm f}_{Q_{P+1}}$ and $V^{\rm f}_P$, respectively. If $V^{\rm f}_P>\rho$, implying that (\ref{eq:ff2}) is violated, we let $Q_{P+1}:=Q_P+1$, ${\mathcal Q}_{P+1}:={\mathcal Q}_P\cup\left\{Q_{P+1}\right\}$, and $\Omega^{\rm v}_{P+1}:=\Omega^{\rm v}_P$. Then, the iteration step increases and we solve (\ref{eq:sec4_master}) again. Otherwise, we let $Q_{P+1}:=Q_P$ and ${\mathcal Q}_{P+1}:={\mathcal Q}_P$. Further, we solve the second subproblem (\ref{eq:sub}) for each $l\in{\mathcal L}$. 
The rest of this algorithm works similarly to the algorithm for (\ref{eq:appx}). Specifically, with ${\xi}_{lP}$ and $F_{lP}$ denoting the solution and optimal value of (\ref{eq:sub}), respectively, we define $\Omega^{\rm v}_{l\left(P+1\right)}:=\Omega^{\rm v}_{lP}\cup\left\{{\xi}_{lP}\right\}$, if $F_{lP}>\rho$, and $\Omega^{\rm v}_{l\left(P+1\right)}:=\Omega^{\rm v}_{lP}$, otherwise. If $F_P:=\max_{l\in{\mathcal L}}F_{lP}\leq\rho$, the algorithm stops and $\left({x}^\ast_{1}, A^\ast, {a}^\ast\right)=\left({x}_{1P}, A_P, {a}_P\right)$ is returned as a solution to (\ref{eq:appx2}). Otherwise, the iteration step increases and we solve (\ref{eq:sec4_master}) again. In Algorithm \ref{alg:sec4}, we provide a pseudocode of the algorithm for (\ref{eq:appx2}). 

Meanwhile, it should be noted that the first subproblem (\ref{eq:maxff}) is a max-min problem that is not easy to handle. To solve (\ref{eq:maxff}), we reformulate it as an MILP using the Big M method \cite{bemporad1999control,gabrel2014robust}. First, we rewrite (\ref{eq:maxff}) as
\begin{equation}\label{eq:maxff2}
\max_{{\zeta}\in\left\{0,1\right\}^m} f^{\rm f}\left({x}_{1P},\underline{\xi} + {\zeta}\circ\left(\overline{\xi}-\underline{\xi}\right)\right)
\end{equation}
where each binary vector $\zeta\in\left\{0,1\right\}^m$ corresponds to a vertex of $\Xi$. Subsequently, we take the dual formulation of the inner problem (\ref{eq:ffproblem}) for ${\xi}=\underline{\xi} + {\zeta}\circ(\overline{\xi}-\underline{\xi})$ to obtain a maximization problem with a bilinear objective function in terms of the dual decision variables and $\zeta$. Finally, we linearize the bilinear terms by introducing auxiliary integer variables to obtain the MILP equivalent to (\ref{eq:maxff}). As a result, we can solve the master problem and two subproblems of (\ref{eq:appx2}) as a finite-dimensional MILP or LP problem. 

Problem (\ref{eq:sec4_master}) is a relaxation of (\ref{eq:appx2}) for any iteration step $P$ such that $V^{\rm f}_P>0$ or $F_P>0$. Thus, $L_P$ monotonically converges to the optimal value of (\ref{eq:appx2}). Moreover, as $\lvert{\mathcal V}\left({\Xi}\right)\rvert=\lvert{\mathcal V}\left({\Omega}\right)\rvert<\infty$, the algorithm yields a solution optimal within the optimality tolerance of the off-the-shelf MILP solver in a finite number of iterations.

\begin{algorithm}[t!]
\caption{Algorithm for (\ref{eq:appx2})}\label{alg:sec4}
\begin{algorithmic}
\Require Feasibility tolerance $\rho\geq0$, any ${\xi}^{\rm f}_1\in{\mathcal V}\left(\Xi\right)$ and ${\xi}_{l1}\in{\mathcal V}\left(\Omega\right)$ for each $l\in{\mathcal L}$ 
\Ensure Solution $\left({x}_1^\ast,A^{\ast},{a}^{\ast}\right)$ to (\ref{eq:appx2})
\For{$l\gets 1$ to $L$}
\State $\Omega^{\rm v}_{l1}\gets\left\{{\xi}_{l1}\right\}$
\EndFor
\State{$Q_1\gets1$, ${\mathcal Q}_1\gets\left\{Q_1\right\}$, $F_1\gets \infty$, $P\gets 1$}
\While{$F_P>\rho$}
\State{Solve (\ref{eq:sec4_master}), $\left({x}_1^\ast,A^{\ast},{a}^{\ast}\right)\gets
\left({x}_{1P},A_P,{a}_P\right)$,}
\State{Solve (\ref{eq:maxff})}
\If{$V^{\rm f}_P>\rho$}
\State{$Q_{P+1}\gets Q_P+1$, ${\mathcal Q}_{P+1}\gets{\mathcal Q}_P\cup\left\{Q_{P+1}\right\}$,}
\State{$\Omega^{\rm v}_{P+1}\gets\Omega^{\rm v}_P$, $F_{P+1}\gets F_P$}
\Else
\For{$l\gets 1$ to $L$}
\State{Solve (\ref{eq:sub})}
\If{$F_{lP}>\rho$}
\State{$\Omega^{\rm v}_{l(P+1)}\gets\Omega^{\rm v}_{lP}\cup\left\{{\xi}_{lP}\right\}$}
\Else
\State{$\Omega^{\rm v}_{l(P+1)}\gets\Omega^{\rm v}_{lP}$}
\EndIf
\EndFor
\State{$F_{P+1}\gets\max_l F_{lP}$}
\EndIf
\State{$P\gets P+1$}
\EndWhile
\end{algorithmic}
\end{algorithm}

In the following section, we examine the applicability and effectiveness of the 2-DRLP formulation (\ref{eq:appx2}) using an affine policy for a practical decision-making problem. 

\section{Application to Unit Commitment}\label{sec:UC}
In this section, we develop a UC model in the form of (\ref{eq:appx2}) for power systems under the uncertainty of renewable energy generation (REG). As a fundamental planning problem for conventional generators, the UC problem is solved on a daily basis to optimize their commitment status as well as economic dispatch policies (i.e., the tertiary controllers) given a forecast of the REG and demand. In the following subsections, we first present a deterministic UC model without considering any uncertainty to introduce basic decision variables and constraints. Subsequently, we explain how to formulate our UC model. Finally, we discuss the results of numerical experiments. Some symbols used in the previous sections may be re-defined. 

\subsection{Deterministic UC}
We consider the UC problem for a transmission system of $I$ buses connected by $L$ transmission lines over a planning horizon of $T$ time periods, the indices of which are denoted by $i$, $l$, and $t$, respectively. Each bus has a conventional generator, a load, and an REG system, all with the same index. The demand of the load at each bus in each time period is known a priori, while the REG is uncertain. We define the forecast error of REG at bus $i$ in time period $t$ as a random variable $\xi_{it}$. Let ${\xi}$ denote a vector of $\xi_{it}$ for all $\left(i,t\right)$. The REG curtailment and demand shedding are fully allowed with penalties. The transmission network is represented by a DC power flow model. Assuming that the realization of $\xi$ is given, we formulate a deterministic UC model in this subsection. 

The decision variables of the deterministic UC model are binary variables $u^{\rm o}_{it}$, $u^{\rm u}_{it}$, and $u^{\rm d}_{it}$, denoting the on/off, start-up, and shut-down status of generator $i$ in time period $t$, respectively, and real variables $x^{\rm g}_{it}$, $x^{\rm r}_{it}$, and $x^{\rm d}_{it}$, denoting the conventional generation, REG curtailment, and demand shedding at bus $i$ in time period $t$, respectively. Let $u\in\{0,1\}^{3IT}$ and $x\in{\mathbb R}^{3IT}$ denote vectors of the binary and real decision variables, respectively. 

The objective of the deterministic UC model is to minimize the total operating cost, i.e., a sum of the fixed cost $c^\top_1 u$ and the variable cost $c^\top_2 x$. Here, $c_1\in{\mathbb R}^{3IT}$ is defined by the no-load, start-up, and shut-down costs of each generator. Further, $c_2\in{\mathbb R}^{3IT}$ is defined by the marginal costs of conventional generation, REG curtailment, and demand shedding at each bus in each time period. 

We denote by ${\mathcal U}$ the feasible set of $u$, which is defined by logic constraints among the binary variables as well as the minimum up and down time constraints of each generator. For a specific formulation of $\mathcal U$, the reader is referred to \cite{cho2022affine}. 
The other decision vector $x$ should meet the following constraints (\ref{eq:gen_ul})--(\ref{eq:bal}) for all the associated indices $(i,l,t)$. First, the conventional generation is chosen under the capacity constraint
\begin{equation}\label{eq:gen_ul}
\underline{X}_i u^{\rm o}_{it} \leq {x}^{\rm g}_{it} \leq \overline{X}_iu^{\rm o}_{it}
\end{equation}
where $\underline{X}_i$ and $\overline{X}_i$ denote the minimum and maximum possible output of generator $i$ when it is in operation, respectively, in addition to the ramping constraint
\begin{equation}\label{eq:ramping}
\begin{aligned}
&-X^{\rm rd}_i u^{\rm o}_{it} - X^{\rm sd}_i u^{\rm d}_{it}\leq x^{\rm g}_{it} - x^{\rm g}_{i\left(t-1\right)} \leq X^{\rm ru}_{i} u^{\rm o}_{i(t-1)} + X^{\rm su}_i u^{\rm u}_{it}
\end{aligned}
\end{equation}
where $X^{\rm rd}_i$, $X^{\rm sd}_i$, $X^{\rm ru}_i$, and $X^{\rm su}_i$ denote the ramp-down, shut-down-ramp, ramp-up, and start-up-ramp limits of generator $i$,  respectively. The upper and lower limits of REG curtailment and those of demand shedding 
are expressed by 
\begin{equation}\label{eq:curt_shed_ul}
0\leq x^{\rm r}_{it} \leq w_{it} + {\xi}_{it},\quad
0\leq x^{\rm d}_{it} \leq d_{it}
\end{equation}
where $w_{it}$ and $d_{it}$ denote the forecast of REG and the demand at bus $i$ in time period $t$, respectively. Furthermore, $x$ should never violate two system-wide constraints, i.e., the transmission capacity constraint 
\begin{equation}\label{eq:line}
\begin{aligned}
&-{F}_l\leq \sum\nolimits_{i}F_{il} \left(x^{\rm g}_{it}+ w_{it} + {\xi}_{it} - x^{\rm r}_{it} - d_{it} + x^{\rm d}_{it} \right)\leq{F}_{l}
\end{aligned}
\end{equation}
where $F_l$ and $F_{il}$ denote the maximum possible power flow in transmission line $l$ and the power shift factor between bus $i$ and transmission line $l$, respectively, and the power supply--demand balance condition
\begin{equation}\label{eq:bal}
\sum\nolimits_{i} \left(x^{\rm g}_{it} + w_{it} + {\xi}_{it} - x^{\rm r}_{it} - d_{it} + x^{\rm d}_{it}\right)=0.
\end{equation}
Compactly, the deterministic UC model is written as 
\[
\begin{aligned}
\min_{u\in{\mathcal U}, x\in{\mathbb R}^{3IT}}&&c^\top_1u + c^\top_2x\text{\quad s.t.\ }\text{(\ref{eq:gen_ul})--(\ref{eq:bal})}\quad \forall i,l,t,
\end{aligned}
\]
which is an MILP that can be easily solved using off-the-shelf solvers. However, $\xi$ is unknown a priori in practice. To address the uncertainty of $\xi$, we use the 2-DRLP formulation (\ref{eq:appx2}) as explained in the following subsection.

\subsection{Proposed Model}
Our UC model is obtained by applying Wasserstein DRO and an affine policy to the two-stage robust UC model in \cite{cho2019box} modified for the transmission system of our interest. In this subsection, we formulate the Wasserstein DRO counterpart of the model in \cite{cho2019box}. Subsequently, we explain the affine policy. The complete formulation of the proposed UC model is omitted to avoid redundancy. 

We first provide the Wasserstein DRO counterpart of the UC model in \cite{cho2019box}
\begin{equation}\label{eq:EDRUC}
\min_{u\in{\mathcal U},\left(\overline{x}^{\rm g},\underline{x}^{\rm g}\right)\in{\mathcal X}^{\rm g}\left(u\right)} {c}^\top_1{u} + \max_{{\mathbb P}\in{\mathcal P}_{\varepsilon}\left({\Xi}\right)} {\mathbb E}_{\mathbb P}\left[f\left(\overline{x}^{\rm g},\underline{x}^{\rm g},{\xi}\right)\right]
\end{equation}
where 
\begin{equation}\label{eq:f_uc}
\begin{aligned}
&f\left(\overline{x}^{\rm g},\underline{x}^{\rm g},{\xi}\right):=\min_{{x}\in{\mathcal X}\left(\overline{x}^{\rm g},\underline{x}^{\rm g},{\xi}\right)}{c}^\top_2{x}
\end{aligned}
\end{equation}
denotes the optimal value of the second-stage problem. The entries of $\overline{x}^{\rm g}\in{\mathbb R}^{IT}$ and $\underline{x}^{\rm g}\in{\mathbb R}^{IT}$ are $\overline{x}^{\rm g}_{it}$ and $\underline{x}^{\rm g}_{it}$ for all $\left(i,t\right)$, respectively. Here, $\overline{x}^{\rm g}_{it}$ and $\underline{x}^{\rm g}_{it}$ denote the allowable upper and lower limits of conventional generation at bus $i$ in time period $t$, respectively, which are introduced to enable the non-anticipative operation of each conventional generator. Specifically, $\overline{x}^{\rm g}_{it}$ and $\underline{x}^{\rm g}_{it}$ are designed so that any $x^{\rm g}_{it}$ such that
\begin{equation}\label{eq:gen_ul_in_ft}
\underline{x}^{\rm g}_{it}\leq x^{\rm g}_{it}\leq\overline{x}^{\rm g}_{it}
\end{equation}
can be implemented independently of the conventional generation at bus $i$ in any other time period while satisfying the capacity and ramping constraints of generator $i$. Accordingly, ${\mathcal X}^{\rm g}\left({u}\right)$ is defined as a set of any $\left(\overline{x}^{\rm g},\underline{x}^{\rm g}\right)$ such that the following constraint hold: 
\[
\begin{cases}
\begin{aligned}
&\underline{X}_i u^{\rm o}_{it} \leq\underline{x}^{\rm g}_{it}\leq \overline{x}^{\rm g}_{it} \leq \overline{X}_iu^{\rm o}_{it}\\
&\overline{x}^{\rm g}_{it} - \underline{x}^{\rm g}_{i\left(t-1\right)} \leq X^{\rm ru}_{i} u^{\rm o}_{i(t-1)} + X^{\rm su}_i u^{\rm u}_{it}\\
&\overline{x}^{\rm g}_{i\left(t-1\right)} - \underline{x}^{\rm g}_{it}  \leq X^{\rm rd}_i u^{\rm o}_{it} + X^{\rm sd}_i u^{\rm d}_{it}
\end{aligned}
\end{cases}
\forall i,t.
\]
Given $\left(\overline{x}^{\rm g},\underline{x}^{\rm g}\right)\in{\mathcal X}^{\rm g}\left({u}\right)$, the feasible set of ${x}$ in (\ref{eq:f_uc}) is defined as
\[
\begin{aligned}
&{\mathcal X}\left(\overline{x}^{\rm g},\underline{x}^{\rm g},{\xi}\right):=\left\{{x}\in{\mathbb R}^{3IT}:\text{(\ref{eq:gen_ul_in_ft}),(\ref{eq:curt_shed_ul})--(\ref{eq:bal})} \quad \forall i,l,t\right\},
\end{aligned}
\]
which, notably, encodes no dynamic constraint. Thus, for a fixed $\left(\overline{x}^{\rm g},\underline{x}^{\rm g}\right)$, solutions to (\ref{eq:f_uc}) for time period $t$ depend only on the realization of forecast error in time period $t$. In other words, we can optimize the conventional generation, REG curtailment and demand shedding at the second stage {\it non-anticipatively}, i.e., not using the future realization of forecast error. If not relying on $\left(\overline{x}^{\rm g},\underline{x}^{\rm g}\right)$, then the ramping constraint (\ref{eq:ramping}) may  still be effective at the second stage. This implies that we have to observe the future forecast error to determine the conventional generation, REG curtailment and demand shedding in each time period, which is unrealistic. Thus, we introduce and determine $\left(\overline{x}^{\rm g},\underline{x}^{\rm g}\right)$ at the first stage. Meanwhile, the support $\Xi$ of $\xi$ is defined using the forecast of REG as well as the capacity of each REG system.

We now apply an affine policy to (\ref{eq:EDRUC}). In this study, we use $x^{\rm ga}_{it}\left({\xi}^{\rm t}_t\right) := a^{\rm g}_{it}\xi^{\rm t}_t + b^{\rm g}_{it}$, $x^{\rm da}_{it}\left({\xi}^{\rm t}_t\right) := a^{\rm d}_{it}\xi^{\rm t}_t + b^{\rm d}_{it}$, and $x^{\rm ra}_{it}\left({\xi}^{\rm t}_t\right) := a^{\rm r}_{it}\xi^{\rm t}_t + b^{\rm r}_{it}$ as decision rules for $x^{\rm g}_{it}$, $x^{\rm d}_{it}$, and $x^{\rm r}_{it}$ for each $(i,t)$, respectively, where $\xi^{\rm t}_t:=\sum_{i}\xi_{it}$ denotes the total forecast error in time period $t$. Although the coefficients of an affine function can be arbitrary as discussed in Section \ref{sec:appx}, we employ these functions to reduce the number of decision variables. Similar affine functions are frequently adopted in the literature on two-stage optimization for power system operations \cite{lorca2016multistage,duan2018distributionally}. 

Applying the affine policy to (\ref{eq:EDRUC}) and, further, re-defining the Wasserstein ball over $\Omega\subseteq\Xi$, we can formulate our UC model in the form of (\ref{eq:appx2}). In the following subsection, we discuss simulation results. 

\subsection{Numerical Experiments}\label{sec:numericalsimulations}
In this section, we compare the economic and computational performances of our UC model to those of six existing models, SUC, RUC, MUC, KUC, NUC and CUC, on  6-,  24-, and  118-bus test systems. Here, SUC and RUC are the SP and RO counterparts of (\ref{eq:EDRUC}), respectively. Moreover, MUC, KUC, NUC and CUC are modifications of the UC models using DRO with ambiguity sets based on the moment conditions, KL divergence, 1-norm distance, and CDF in \cite{zhou2019distributionally}, \cite{chen2018distributionally}, \cite{ding2018duality}, and \cite{duan2017data}, respectively. 

The generator, load, and branch data of the 6- and 24-bus systems are  from \cite{IIT} and \cite{ordoudis2016updated}, respectively. We locate one wind farm of capacity 80 MW at bus 2 of the 6-bus system, and three wind farms, each of capacity 300 MW, at buses 3, 5, and 7 of the 24-bus system. For the 118-bus system, the generator and load data are from \cite{IIT}, and we use the branch data from \cite{pena2017extended} to accommodate five wind farms of capacities 40 MW, 75 MW, 120 MW, 250 MW, and 300 MW at buses 24, 27, 31, 100, and 82, respectively, as well as five solar farms of capacities 700 MW, 330 MW, 200 MW, 200 MW, and 150 MW at buses 32, 92, 54, 18, and 15, respectively. The penetration levels of REG  (i.e., the ratio of the total REG capacity to the peak demand) of the 6-, 24-, and 118-bus systems are 30.77\%,  33.96\%, and 35.38\%, respectively. For all the test systems, we use a planning horizon of $T=24$ time periods of 1h. We assume that no more than 10 loads with the highest total demand can be shed, while all the REG systems can be curtailed. The marginal costs of demand shedding and REG curtailment are set to \$3500/MWh and \$20/MWh, respectively. We run the simulations using MATLAB with MOSEK 9.3 for MUC and using CPLEX 12.10 for the others on a PC with an Intel Core i7 3.70 GHz processor and 32 GB RAM. We discuss the results in the following subsections.

 \begin{figure}[t!]
    \centering
       \includegraphics[width=6in]
       {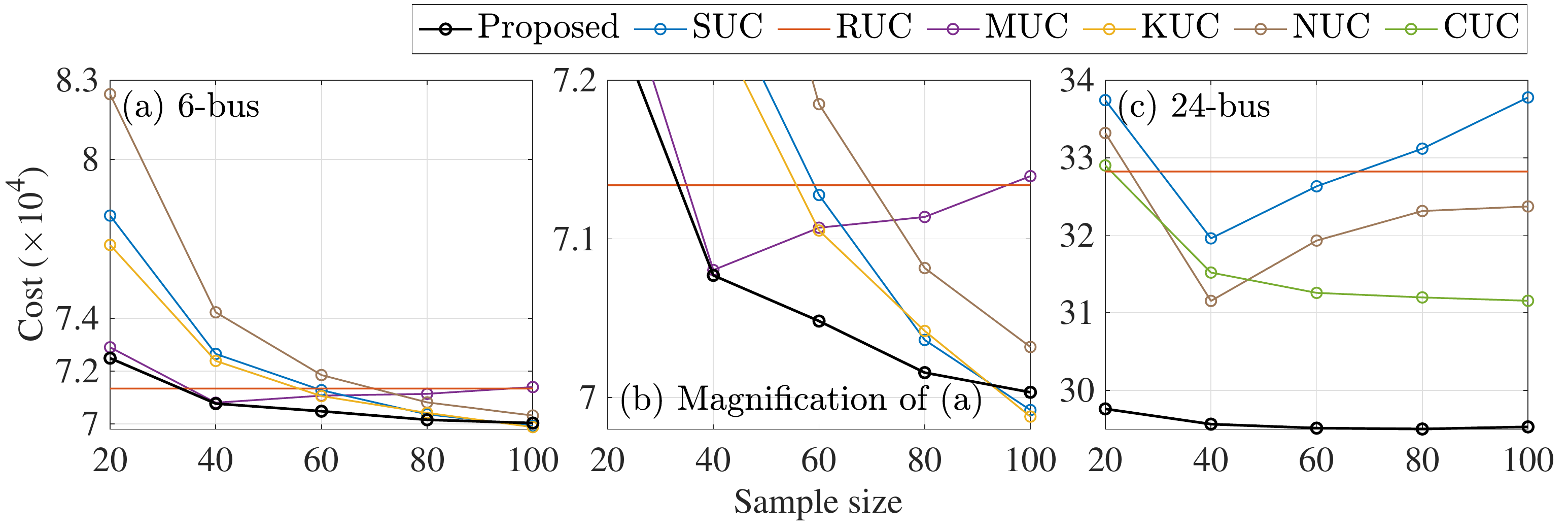}
    \caption{Average out-of-sample costs for different sample sizes.}
\label{fig:DRO_6bus}
\end{figure}

\subsubsection{Comparison via Random Sampling}
In the following, we compare our UC model to the six benchmark models on the 6- and 24-bus systems via random sampling. The simulation scheme is as follows: First, we model the true distribution ${\mathbb P}^\star$ of the wind power forecast error as a Pearson distribution based on the observation data from  \cite{hodge2012wind}. Randomly generating $N=20, 40, \ldots, 100$ samples according to ${\mathbb P}^\star$, we build empirical distributions of the forecast error and solve each UC model. We repeat this process 50 times, i.e., with 50 independent sample sets, for statistical robustness. Then, we compare the UC models in terms of the average out-of-sample cost and average computation time. For a here-and-now decision $\left({u},\overline{x}^{\rm g},\underline{x}^{\rm g}\right)\in{\mathcal U}\times{\mathcal X}^{\rm g}\left({u}\right)$ obtained by solving any model, the {\it out-of-sample cost} is defined as 
\[
J\left({u},\overline{x}^{\rm g},\underline{x}^{\rm g}\right):={c}_1^\top{u} + {\mathbb E}_{{\mathbb P}^\star}\left[f\left(\overline{x}^{\rm g},\underline{x}^{\rm g},{\xi}\right)\right].
\]
As exactly computing the out-of-sample cost is difficult, we use the sample average approximation to estimate it with an additional 10,000 scenarios of the forecast error that are randomly generated according to ${\mathbb P}^\star$ independently of the $N$ samples. For our UC model, we set $\beta=100$ and use the holdout method \cite{esfahani2018data} to choose $\varepsilon$ from $10^{-3}$, $10^{-2}$ and $10^{-1}$. For each benchmark model using DRO, we set the parameter(s) of the ambiguity set as guided in the corresponding research paper with its confidence level, if required, set to 0.99. We set a timeout limit of 1h only for the 24-bus system. Further, we solve SUC, KUC and NUC on the 24-bus system with the fast-forward selection method \cite{heitsch2003scenario} to reduce the number of samples used for building empirical distributions, thus avoiding time-out and memory-outage errors.

On the 6-bus system, MUC has no solution with one sample set for $N=20$, while CUC is infeasible with $35,42,45,47,49$ sample sets for $N=20,\ldots,100$, respectively. On the 24-bus system, MUC and KUC with the first five sample sets for any $N$ cannot be solved due to timeout errors, neither of which we implement further. We illustrate the average out-of-sample costs of each UC model in Fig. \ref{fig:DRO_6bus}, except those of CUC for the 6-bus system and those of MUC and KUC for the 24-bus system. We also report the average computation time on the 6- and 24-bus system in Tables \ref{tab:6bus_time} and \ref{tab:24bus_time}, respectively. 

\renewcommand{\arraystretch}{1}\setlength{\tabcolsep}{0.5em}
\begin{table}[t]
\center
\caption{Average computation time (in seconds): 6-bus system}\label{tab:6bus_time}
\begin{tabular}{ |c|c c c c c| }
\hline
$N$ & 20 & 40 & 60 & 80 & 100 \\
\hline
\hline
Prop. & 8.04 & 5.74 & 7.81 & 9.56 & 6.40 \\
\hline
RUC & \multicolumn{5}{c|}{8.22} \\
\hline
SUC & 1.96 & 3.98 & 7.69 & 12.41 & 19.55 \\
\hline
NUC & 16.62 & 45.78 & 93.57 & 156.22 & 247.53\\
\hline
MUC & 310.79 & 365.57 & 372.34 & 356.28 & 343.87 \\
\hline
KUC & 3834.08 & 5966.40 & 6106.50 & 6665.26 & 6761.11 \\
\hline
\end{tabular}
\end{table}

\renewcommand{\arraystretch}{1}\setlength{\tabcolsep}{0.74em}
\begin{table}[t]
\center
\caption{Average computation time (in seconds): 24-bus system}\label{tab:24bus_time}
\begin{tabular}{ |c|c c c c c| }
\hline
$N$ & 20 & 40 & 60 & 80 & 100 \\
\hline
\hline
Prop. & 77.60 & 73.57 & 76.11 & 106.58 & 79.63 \\
\hline
RUC & \multicolumn{5}{c|}{2.99} \\
\hline
CUC & 9.44 & 10.65 & 11.04 & 10.96 & 10.85\\
\hline
SUC & 19.26 & 193.01 & 140.87 & 184.96  & 137.94\\
\hline
NUC & 79.68 & 765.98 & 761.52 & 763.93  & 618.36 \\
\hline
\end{tabular}
\end{table}

In Fig. \ref{fig:DRO_6bus}, the proposed model shows the lowest average out-of-sample costs for $N=40,60,80$ and all $N$'s on the 6- and the 24-bus system, respectively. For $N=20$ on the 6-bus system, RUC leads to the lowest average out-of-sample cost. Thus, RUC, which is the most robust, can be an alternative to our model when there are few samples. For $N=100$ on the 6-bus system, SUC and KUC perform better than the proposed model. In fact, SUC, KUC and NUC may incur lower out-of-sample costs than our model when a huge number of samples are available. However, it is highly likely that their computational performances are not satisfactory even for moderate-size systems in such a case due to their poor scalability regarding sample size, as can be observed from Fig. \ref{fig:DRO_6bus} (c).

Tables \ref{tab:6bus_time} and \ref{tab:24bus_time} verify that the computational load of our model is independent of sample size. Although our model is not the most computationally tractable for every single case, the average increase in computation time, if any, is acceptable given the accompanying decrease in the average out-of-sample cost for most of the cases, compared to any benchmark model.

\subsubsection{Comparison Using Real Data}
In the following, we further compare the UC models on the 118-bus system with a 365-day real-world data set. The data sets of wind and solar power forecast errors are from \cite{canwea} and \cite{nrel}, respectively. The simulation scheme is as follows: First, we construct $S_N$ pairs of training and test distributions, both of which are empirical distributions obtained using $N$-day observation data before and from day $D^N_{k}$, $k = 1,2,\ldots,S_N$ of the year, respectively. We set $S_N=11,9,\ldots,1$ for $N=30,60,\ldots,180$, respectively. Further, day $D^N_k$ corresponds to the first day of month $N/30+k$, except for $\left(N,k\right)=\left(60,1\right)$, in which case we set $D^{60}_1=61$ to represent the 2nd of March. We build the distribution pairs in this way so they consecutively cover almost all of the one-year observation data. We solve each UC model for each training distribution and evaluate the ``cost," i.e., the expected total operating cost with respect to the associated test distribution, as well as the computation time. For SUC, KUC and NUC, we rebuild the training distributions with only five samples obtained using the scenario reduction method. We set a timeout limit of 3h. 

Tables \ref{tab:118cost} and \ref{tab:118time}  show the average costs and computation times of the UC models, except for MUC and KUC, which face memory-outage and timeout errors for all  cases, respectively. Moreover, the results of CUC are only for six and three distribution pairs with $N=30,60$, respectively, except when it is infeasible. The numbers in parentheses are the percentage increases from the average costs of our model to those of RUC, which is the closest to our model in terms of the average cost. The results indicate that the proposed model leads to the lowest average costs for all cases except for the smallest $N$ at the expense of acceptable increases in computation time, similar to the results for the 6- and 24-bus systems, but on the larger-scale system with the real data set. 

\renewcommand{\arraystretch}{1.1}\setlength{\tabcolsep}{0.15em}
\begin{table}[t!]
\centering
\caption{Average cost ({$\$10^6$}): 118-bus system}\label{tab:118cost}
\begin{tabular}{ |c|c c c c c c| }
\hline
$N$ & 30 & 60 & 90 & 120 & 150 & 180 \\
\hline
\hline
Prop. & 1.34 & 1.31 & 1.33 & 1.33 & 1.34 & 1.37 \\
\hline
SUC & 3.01 & 2.88 & 3.05 & 3.19 & 3.00 & 2.90 \\
\hline
\multirow{2}{*}{RUC} & 1.33 & 1.33 & 1.35 & 1.34 &  1.35 & 1.38 \\
 & ($-$0.51\%) & (1.77\%) & (1.59\%) & (1.49\%) &  (1.21\%) & (0.99\%) \\
\hline
NUC & 3.00 & 3.04 & 3.07 & 3.17 & 3.09 & 2.59 \\
\hline
CUC & 1.49 & 1.36 & - & - & - & -\\
\hline
\end{tabular}
\end{table}

\renewcommand{\arraystretch}{1.1}\setlength{\tabcolsep}{0.378em}
\begin{table}[t!]
\centering
\caption{Average computation time (in seconds):  118-bus system}\label{tab:118time}
\begin{tabular}{ |c|c c c c c c| }
\hline
$N$ & 30 & 60 & 90 & 120 & 150 & 180 \\
\hline
\hline
Prop. & 418.59 & 634.39 & 457.30 & 369.49 & 405.62 & 364.09 \\
\hline
SUC & 113.31 & 100.42 & 127.74 & 211.74 & 141.07 & 133.59 \\
\hline
RUC & \multicolumn{6}{c|}{120.80} \\
\hline
NUC & 313.81 & 217.56 & 227.70 & 173.65 & 139.94 & 135.23 \\
\hline
CUC & 2571.70 & 1597.13 & - & - & - & - \\
\hline
\end{tabular}
\end{table}

\section{Conclusions}\label{sec:conclusions}
In this article, we studied a generic class of 2-DRLPs over 1-Wasserstein balls using affine policies. We showed that the problem of our interest has a tractable reformulation with a scale independent of sample size. Subsequently, we proposed a method for refining the Wasserstein ball to reduce the conservativeness of affine policies. To examine the effectiveness of the 2-DRLP formulation with an affine policy, we also developed a novel UC model for power systems under uncertainty and conducted extensive numerical experiments. Future research directions include analyzing the suboptimality of affine policies with additional assumptions on the problem structure and extending our study to multi-stage settings.

\bibliographystyle{IEEEtran}

\bibliography{reference}

\end{document}